\documentclass{article}
\usepackage{amsmath, amsthm}
  \usepackage{paralist}
  \usepackage{graphics} 
  \usepackage{epsfig} 
\usepackage{graphicx}  \usepackage{epstopdf}
 \usepackage[colorlinks=true]{hyperref}
 \usepackage{amsfonts,amssymb}
\hypersetup{urlcolor=blue, citecolor=red}

     \makeatletter
\providecommand{\@LN}[2]{}
\makeatother

\newtheorem{theorem}{Theorem}[section]

\newtheorem{lemma}[theorem]{Lemma}

\newtheorem{remark}{Remark}

\title 
      {Using a kinetic BGK model to determine transport coefficients of gas mixtures}

\date{}

\begin{document}
\maketitle

\centerline{\scshape Christian Klingenberg and Marlies Pirner$^*$}
\medskip
{\footnotesize
 \centerline{University of W\"urzburg, }
   \centerline{ Emil-Fischer-Str. 40, }
   \centerline{ 97074 W\"urzburg, Germany}
} 



\bigskip


\begin{abstract}
We consider a non reactive two component gas mixture. In a macroscopic description of a gas mixture we expect  four physical coefficients characterizing the physical behaviour of the gases to appear. These are the diffusion coefficient, the viscosity coefficient, the heat conductivity and the thermal diffusion parameter in the Navier-Stokes equations. We present a Chapman-Enskog expansion of a kinetic model for gas mixtures by Klingenberg, Pirner and Puppo, 2017 that has three free parameters in order to capture three of these four physical coefficients. In addition, we propose several possible extensions to an ellipsoidal statistical model for gas mixtures in order to capture the fourth coefficient. 
\end{abstract}

\\ 

\textbf{keywords:} multi-fluid mixture, kinetic model, ES-BGK approximation, transport coefficients, Chapman-Enskog expansion

\section{Introduction}
 In this paper we shall concern ourselves with a kinetic description of gases. This is traditionally done via the Boltzmann equation for the density distributions $f_1$ and $f_2$. Under certain assumptions the complicated interaction terms of the Boltzmann equation can be simplified by a so called BGK approximation (approximation invented by Bhatnagar, Gross and Krook in \cite{Bhatnagar}), consisting of a collision frequency multiplied by the deviation of the distributions from local Maxwellians.  This approximation should be constructed in a way such that it  has the same main properties of the Boltzmann equation namely conservation of mass, momentum and energy, further it should have an H-theorem with its entropy inequality and the equilibrium must still be Maxwellian.  BGK  models give rise to efficient numerical computations, which are asymptotic preserving, that is they remain efficient even approaching the hydrodynamic regime \cite{Puppo_2007,Jin_2010,Dimarco_2014,Bennoune_2008,Bernard_2015,Crestetto_2012}. However, the drawback of the BGK approximation  is its incapability of reproducing the correct Boltzmann hydrodynamic regime in the asymptotic continuum limit. Therefore, a modified version called ellipsoidal statistical model (ES-BGK approximation) was suggested  by Holway in the case of one species \cite{Holway}. The H-Theorem of this model then was proven in \cite{Perthame} and existence and uniqueness of solutions in \cite{Yun}.
 
 Here we shall focus on gas mixtures modelled via an ES-BGK approach. In the literature there is a BGK model for gas mixtures  suggested by Andries, Aoki and Perthame in \cite{AndriesAokiPerthame2002} which contains only one collision term on the right-hand side. In \cite{AndriesAokiPerthame2002} the hydrodynamic limit obtained by performing the Chapman-Enskog expansion is considered. Since the model does lead to the right hydrodynamic regime, one extension of this model to an ES-BGK model to an ES-BGK model for gas mixtures is given by Brull in \cite{Brull}. His extension is based on an entropy minimization problem and leads to a correct Prandtl number in the Navier-Stokes equations. Another extension is given by Groppi, Monica and Spiga in \cite{Groppi}. They noticed that in the case of a gas mixture there is not only one physical quantity, the Prandtl number, which they want to fix in the right way, there are also the diffusion coefficient and the thermal diffusion parameter. In \cite{Groppi}, with the proposed model there, they are able to fix the diffusion parameter and the Prandtl number but not the thermal diffusion parameter. There are also hydrodynamic limits for kinetic models dealing with reactive mixtures, and for mixtures related to the description of polyatomic molecules, see for example \cite{Bisi_2017,Bisi_2011,Kremer_2006} and references therein.


 In \cite{Pirner}, Klingenberg, Pirner and Puppo present a BGK model for gas mixtures which contains three free parameters from a modelling point of view. In this paper we are interested in an extension to an ES-BGK model of this BGK model for gas mixtures \cite{Pirner}. The advantage of this model is that we already have free parameters that open the possibility to determine macroscopic physical quantities like viscosity or heat conductivity when performing the Chapman-Enskog expansion  to obtain the Navier-Stokes equations.


The outline of this chapter is as follows: in section \ref{sec7.1} we want to motivate the ES-BGK model for one species. In section \ref{sec7.3} we briefly repeat the main issues of the BGK model for gas mixtures in \cite{Pirner}. In section \ref{sec7.3.6.1}, we  introduce the macroscopic equations and quantities Groppi, Monica and Spiga \cite{Groppi} expect to have in the case of gas mixtures. In section  \ref{sec7.5} the Chapman-Enskog expansion for the BGK model for mixtures from section \ref{sec7.3} is performed  in order to arrive at the transport coefficients. Our BGK model used for this derivation has free parameters that are shown to be useful in this derivation. 
In section \ref{sec7.4}, we suggest extensions to an ES-BGK model for mixtures.
\section{Motivation of the ES-BGK model for one species}
\subsection{The BGK and ES-BGK model for one species }
\label{sec7.1}
In  this section we briefly recall the BGK model for one species and the extension to an ES-BGK model in the case of one species proposed by \cite{Holway} in order to see the effect of the ES-BGK extension. This is presented in more details in \cite{AndriesPerthame2001}. \\
For one species the BGK equation is given by
$$\partial_t f + v \cdot \nabla_x f = \nu n (M(f)-f),$$
with the collision frequency $\nu(x,t)$, the distribution function $f(x,v,t)>0$ where $x \in \mathbb{R}^3$, $v \in \mathbb{R}^3$ are the phase space variables and $t\geq 0$ the time.
We relate the distribution functions to  macroscopic quantities by mean-values of $f$ via
\begin{align}
\begin{split}
\int f(v) \begin{pmatrix}
1 \\ v \\ m |v-u|^2 \\ 
\end{pmatrix} 
dv =: \begin{pmatrix}
n \\ n u \\ 3 n T 
\end{pmatrix},
\end{split}
\label{macrosquone}
\end{align}
where $m$ is the mass of a particle, $n$ is the number density, $u$ the mean velocity and $T$ the temperature which is related to the pressure $p$ by $p=n T$. Note that in this paper we shall write $T$ instead of $k_B T$, where $k_B$ is Boltzmann's constant. Then we can define the Maxwellian $M(f)$ by
\begin{align} 
M(x,v,t) = \frac{n}{\sqrt{2 \pi \frac{T}{m}}^3} \exp \left({- \frac{|v-u|^2}{2 \frac{T}{m}}} \right),
\end{align}
For the ES-BGK model we now replace the Maxwell distribution $M(f)$ by another relaxation operator and consider the new equation
$$\partial_t f + v \cdot \nabla_x f = \nu n (G(f)-f),$$
 where $G(f)$ is not a Maxwell distribution any more. Instead of the scalar temperature in the Maxwell distribution we take a linear combination of the temperature and the pressure tensor: 
\begin{align}
 G(f) = \frac{n}{\sqrt{\det ( 2 \pi \frac{\mathcal{T}}{m}})} e^{- \frac{1}{2} (v-u) \cdot (\frac{\mathcal{T}}{m})^{-1} \cdot (v-u)},
 \label{GES}
 \end{align}
  where $$\mathcal{T}= (1- \tilde{\mu}) T \textbf{1} + \tilde{\mu} \frac{\mathbb{P}}{n}, \quad \text{with} \quad - \frac{1}{2} \leq \tilde{\mu} \leq 1,$$ being a free parameter. 
We see that for $\tilde{\mu}=0$ we regain the BGK model because then $\mathcal{T}^{-1}= \frac{1}{T} \textbf{1}$ and $\det (2 \pi \frac{T}{m} \textbf{1})= (2 \pi \frac{T}{m})^3$ in three space dimensions.

Since we wrote $\mathcal{T}^{-1}$ we have to ensure that $\mathcal{T}$ is invertible. This is done by the next lemma and the next theorem.
\begin{lemma}
Assume that $f>0$. Then $\frac{{\mathbb{P}}}{n }$  has strictly positive eigenvalues.
\end{lemma}
The proof  is given in \cite{Pirner2} in the case of two species.
\begin{theorem}
Assume that $f>0$ and $ - \frac{1}{2} \leq \tilde{\mu} \leq 1.$ Then $\mathcal{T}$ has strictly positive eigenvalues. Especially $\mathcal{T}$ is invertible.
\end{theorem}
The proof is given by Andries and Perthame in \cite{AndriesPerthame2001}. 
\subsection{The theory of persistence of the velocity}
\label{sec7.1.3}
In  section \ref{sec7.4} we want to propose possible extensions to an ES-BGK model for gas mixtures. One of these extensions is based on the following physical theory. It is called  theory of the persistence of velocities. It is described in \cite{Jeans,Jeans2,Heer,Chapman}. The theory of persistence of velocities after a collision of two particles is a physical phenomenon developed by Jeans in \cite{Jeans}. It states the following. After a collision with another particle the velocity of a given particle will, on the average, still retain a component in the direction of its original motion. This is explained in the following. Jeans computes a mean absolute value   $\bar{c}'_1(c_1, c_2)$ of the velocity after the collision of a particle with mass $m_1$ which has the absolute value of $c_1$ of its velocity before collision via
\begin{align}
 \bar{c}_1'(c_1,c_2)= \frac{ \int_{S^2} c_1' \tilde{\nu}_{12}(c_2, \omega) d\omega}{\int_{S^2} \tilde{\nu}_{12} (c_2, \omega) d\omega },
 \label{velbarc1prime}
 \end{align}
 The index $2$ denotes another particle with mass $m_2$ and a fixed but arbitrary value of the absolute value $c_2$ of the velocity of particle $2$. The average is taken over all possible deflection angles $\omega \in S^2$ and
 $\tilde{\nu}_{12}(c_2, \omega)$ denotes the probability of a collision. 
\subsubsection{The case of equal masses}
In the case of hard balls and equal masses one can compute the integral in \eqref{velbarc1prime}. For details of this computation see \cite{Jeans,Jeans2,Chapman}.
We obtain 
$$\bar{c}_1'(c_1,c_2) = \begin{cases} \frac{15 c_1^4 + c_2^4}{10 c_1 ( 3 c_1^2 + c_2^2)} \quad c_1>c_2, \\ \frac{c_1(5 c_2^2 + 3 c_1^2) }{5 ( 3 c_2^2 + c_1^2)} \quad c_1<c_2. \end{cases}$$ Then the expression $\frac{\bar{c}_1'(c_1,c_2)}{c_1}$ is called the measure of persistence of the velocity of the first particle. We observe that $\frac{\bar{c}_1'(c_1,c_2)}{c_1}$  depends only on the ratio $\kappa:= \frac{c_1}{c_2}$, namely $$ \frac{\bar{c}_1'}{c_1}(\kappa)= \begin{cases} \frac{15 \kappa^4 + 1}{10 \kappa^2 ( 3 \kappa^2 + 1)} \quad \kappa>1, \\ \frac{3 \kappa^2 + 5 }{5 ( \kappa^2 + 3)} \quad \quad  \kappa <1. \end{cases}$$
We observe the following estimate
\begin{lemma}
$\frac{\bar{c}_1'}{c_1}$ satisfies the following estimate $$\frac{1}{4} \leq \frac{\bar{c}_1'}{c_1}(\kappa) \leq 1,$$ 
for all $\kappa \in \mathbb{R}^+$.
\label{perstistest}
\end{lemma}
%
%
%
%
\begin{proof}
First, let $\kappa >1$. Then $\frac{\bar{c}_1'}{c_1}(\kappa)$ is bounded from above by a number smaller $1$ since
$$ \frac{15 \kappa^4 + 1}{10 \kappa^2 ( 3 \kappa^2 + 1)} \leq \frac{15 \kappa^4 + 1}{30 \kappa^4} = \frac{1}{2} + \frac{1}{30} \frac{1}{\kappa^4} \leq \frac{1}{2} + \frac{1}{30},$$
and from below by $\frac{1}{3}$ since
$$ \frac{15 \kappa^4 + 1}{10 \kappa^2 ( 3 \kappa^2 + 1)} \geq \frac{15 \kappa^4 + 1}{10 \kappa^2 ( 3 \kappa^2 + \kappa^2)} = \frac{15 \kappa^4}{40 \kappa^4} = \frac{3}{8} \geq \frac{1}{4}.$$
Next, let $\kappa <1$, then we have
$$ \frac{3 \kappa^2 + 5}{5(\kappa^2 +3)} \leq \frac{3 \kappa^2 +5}{15} \leq 1,$$
and $$ \frac{3 \kappa^2 + 5}{5( \kappa^2 +3)} \geq \frac{5}{20} = \frac{1}{4}.$$
\end{proof}
We observe that for any $\kappa$ we expect a positive persistence of the velocity before collision which is always larger than $\frac{1}{4}$ for any $\kappa$.
\begin{remark}
In the case of different masses we obtain from the integral in \eqref{velbarc1prime} the expression
$$ \frac{\bar{c}_1'}{c_1} = \frac{m_1-m_2}{m_1+m_2} + \frac{2 m_2}{m_1+m_2} \left(\frac{\bar{c}_1'}{c_1}\right)_e,$$
where $(\frac{\bar{c}_1'}{c_1})_e$ denotes the persistence when the masses are equal. So we obtain the following inequality
$$ \frac{\bar{c}_1'}{c_1} \geq \frac{m_1 - \frac{1}{2} m_2}{m_1 +m_2},$$
if we use lemma \ref{perstistest} for the persistence when the masses are equal. So in the case of different masses we observe that it is dependent from the masses whether we have a persistence or not.
\end{remark}
\subsubsection{Consequences for the choice of the ES-BGK operator}
The choice of the tensor $\mathcal{T}= (1- \tilde{\mu}) T \textbf{1} + \tilde{\mu} \mathbb{P}$ can be motivated with the theory of persistence of the velocities as follows. This was done by Holway in \cite{Holway} who was the physicist who invented the ES-BGK model. The theory of persistence of the velocity argues that in the post-collisional absolute value of the velocity there is a memory of the pre-collisional absolute value of the velocity of the particle. In the single species BGK equation this yields to the choice of
$$\mathcal{T} = (1- \tilde{\mu}) T \textbf{1} + \tilde{\mu} \mathbb{P}, \quad -\frac{1}{2} \leq \tilde{\mu} \leq 1,$$
the tensor chosen in the well-known ES-BGK model, where $\tilde{\mu} \mathbb{P}$ preserves the memory of the off-equilibrium content of the pre-collisional velocity. This can be rewritten as
$$ \mathcal{T} = T \textbf{1} + \tilde{\mu}  ~ \text{traceless}[\mathbb{P}], $$
where $\text{traceless}[\mathbb{P}]$ denotes the traceless part of $\mathbb{P}$. So the off-equilibrium part is contained in $\tilde{\mu} ~ \text{traceless}[\mathbb{P}].$ 
\section{The BGK approximation for gas mixtures}
\label{sec7.3}
For simplicity in the following we consider a mixture composed of two different species, but the discussion can be generalized to a multi species mixtures. Thus, our kinetic model has two distribution functions $f_1(x,v,t)> 0$ and $f_2(x,v,t) > 0$ where $x\in \Lambda \subset \mathbb{R}^3$ and $v\in \mathbb{R}^3$ are the phase space variables and $t\geq 0$ the time.  

 Furthermore, for any $f_1,f_2: \Lambda \subset \mathbb{R}^3 \times \mathbb{R}^3 \times \mathbb{R}^+_0 \rightarrow \mathbb{R}$ with $(1+|v|^2)f_1,(1+|v|^2)f_2 \in L^1(\mathbb{R}^3), f_1,f_2 \geq 0$ we relate the distribution functions to  macroscopic quantities by mean-values of $f_k$, $k=1,2$
\begin{align}
\int f_k(v) \begin{pmatrix}
1 \\ v  \\ m_k |v-u_k|^2 \\ m_k (v-u_k(x,t)) \otimes (v-u_k(x,t)) \end{pmatrix} 
dv =: \begin{pmatrix}
n_k \\ n_k u_k \\  3 n_k T_k \\ \mathbb{P}_k
\end{pmatrix} , \quad k=1,2,
\label{moments5}
\end{align} 
where $n_k$ is the number density, $u_k$ the mean velocity, $T_k$ the mean temperature and $\mathbb{P}_k$ the pressure tensor of species $k$, $k=1,2$. Note that in this section we shall write $T_k$ instead of $k_B T_k$, where $k_B$ is Boltzmann's constant.

We are interested in a BGK approximation of the interaction terms. This leads us to define equilibrium distributions not only for each species itself but also for the two interspecies equilibrium distributions. We choose the collision terms   as BGK operators. 
Then the model can be written as:

\begin{align} \begin{split} \label{BGK5}
\partial_t f_1 + v \cdot \nabla_x  f_1   &= \nu_{11} n_1 (M_1 - f_1) + \nu_{12} n_2 (M_{12}- f_1),
\\ 
\partial_t f_2 + v \cdot \nabla_x  f_2 &=\nu_{22} n_2 (M_2 - f_2) + \nu_{21} n_1 (M_{21}- f_2), 
\end{split}
\end{align}
with the Maxwell distributions
\begin{align} 
\begin{split}
M_k(x,v,t) &= \frac{n_k}{\sqrt{2 \pi \frac{T_k}{m_k}}^3 }  \exp \left({- \frac{|v-u_k|^2}{2 \frac{T_k}{m_k}}} \right),
\quad k=1,2,
\\
M_{kj}(x,v,t) &= \frac{n_{kj}}{\sqrt{2 \pi \frac{T_{kj}}{m_k}}^3 }  \exp \left({- \frac{|v-u_{kj}|^2}{2 \frac{T_{kj}}{m_k}}} \right), \quad k,j=1,2,~ k \neq j,
\end{split}
\label{BGKmix5}
\end{align}
where $\nu_{11} n_1$ and $\nu_{22} n_2$ are the collision frequencies of the particles of each species with itself, while $\nu_{12}$ and $\nu_{21}$ are related to interspecies collisions.
To be flexible in choosing the relationship between the collision frequencies, we now assume the relationship
\begin{align} 
\begin{split}
\nu_{12}&= \varepsilon \nu_{21}, \hspace{3.1cm} 0 < \varepsilon \leq 1,
 \\
\nu_{11} &= \beta_1 \nu_{12}, \quad \nu_{22} = \beta_2 \nu_{21}, \quad \quad \beta_1, \beta_2 >0.
\end{split}
\label{coll5}
\end{align}

%
%
%
%
%
%
If we assume that \begin{align} n_{12}=n_1 \quad \text{and} \quad n_{21}=n_2,  
\label{density5} 
\end{align}
 \begin{align}
u_{12}= \delta u_1 + (1- \delta) u_2, \quad \delta \in \mathbb{R},
\label{convexvel5}
\end{align} 
and
\begin{align}
\begin{split}
T_{12} &=  \alpha T_1 + ( 1 - \alpha) T_2 + \gamma |u_1 - u_2 | ^2,  \quad 0 \leq \alpha \leq 1, \gamma \geq 0 ,
\label{contemp5}
\end{split}
\end{align}
we have conservation of the number of particles, of total momentum and total energy provided that
\begin{align}
u_{21}=u_2 - \frac{m_1}{m_2} \varepsilon (1- \delta ) (u_2 - u_1),
\label{veloc5}
\end{align}
and
\begin{align}
\begin{split}
T_{21} =\left[ \frac{1}{3} \varepsilon m_1 (1- \delta) \left( \frac{m_1}{m_2} \varepsilon ( \delta - 1) + \delta +1 \right) - \varepsilon \gamma \right] |u_1 - u_2|^2 \\+ \varepsilon ( 1 - \alpha ) T_1 + ( 1- \varepsilon ( 1 - \alpha)) T_2,
\label{temp5}
\end{split}
\end{align}
 see theorem 2.1, theorem 2.2 and theorem 2.3 in \cite{Pirner}.

 We see that without using an ES-BGK extension, we already have three free parameters in \eqref{convexvel5} and \eqref{contemp5} in order to possibly match coefficients like the Fick's constant or the heat conductivity in the Navier-Stokes equations.  
 In order to ensure the positivity of all temperatures, we need to impose restrictions on $\delta$ and $\gamma$, 
 \begin{align}
0 \leq \gamma  \leq \frac{m_1}{3} (1-\delta) \left[(1 + \frac{m_1}{m_2} \varepsilon ) \delta + 1 - \frac{m_1}{m_2} \varepsilon \right],
 \label{gamma5}
 \end{align}
and
\begin{align}
 \frac{ \frac{m_1}{m_2}\varepsilon - 1}{1+\frac{m_1}{m_2}\varepsilon} \leq  \delta \leq 1,
\label{gammapos5}
\end{align}
see theorem 2.5 in \cite{Pirner}.
  This summarizes our kinetic model \eqref{BGK5} in of two species that contains three free parameters. More details can be found in \cite{Pirner}.
\section{Coefficients on the Navier Stokes level}
\label{sec7.3.6.1}
 In the two species case we expect from the H-Theorem (see theorem 2.7 in \cite{Pirner}) that in equilibrium the two species have a common velocity  and a common temperature. In the following, we will denote this two quantities by $\bar{u}$ and $\bar{T}$.  In this section we want to present what parameters in the Navier-Stokes equations for mixtures we want to fix in the two species case. This is described by Groppi, Monica and Spiga in \cite{Groppi}. They expect the following expansion of the velocities according to Fick's law
\begin{align*}
u_s^{\varepsilon} = \bar{u} - \sum_{r=1}^2 D_{sr} \nabla_x n_r, \quad s=1,2,
\end{align*} 
with the four diffusion coefficients bound together by the three independent constraints
\begin{align*}
D_{12}=D_{21}, \quad \sum_{s=1}^2 D_{sr} m_s n_s, \quad r=1,2,
\end{align*}
so the full Fick matrix is determined by only one of its entries. 
In case of the pressure tensor they expect the following expansion
\begin{align*}
\mathbb{P} := \mathbb{P}_1 + \mathbb{P}_2 = (n_1 + n_2) \bar{T} \textbf{1} - \mu (\nabla_x \bar{u} + \nabla_x \bar{u}^T - \frac{2}{3} \nabla_x \cdot \bar{u} \textbf{1}),
\end{align*}
where $\mu$ denotes the viscosity coefficient of the gas mixture. In the case of the derivation of the heat flux, they expect the following
\begin{align*}
\tilde{Q} := \tilde{Q}_1 + \tilde{Q}_2 = - \lambda (m_1 + m_2) \sum_{s=1}^2 \frac{n_s}{m_s} \nabla_x \bar{T} + C \sum_{s=1}^2 n_s (u_s - \bar{u}),
\end{align*}
where $\tilde{Q}_k$ is $\tilde{Q}_k= m_k \int (v-u_k) |v-u_k|^2 f_k dv$, $k=1,2$ and $\lambda$ denotes the heat conductivity and $C$ the thermal diffusion parameter. 
\section{Chapman Enskog expansion for the mixture}
\label{sec7.5}
In this section  perform the Chapman Enskog expansion for the BGK equation for mixtures \eqref{BGK5} in order to see where the free parameters $\alpha, \delta$ and $\gamma$ will show up at the Navier-Stokes level. For simplicity, we do not do all the computations but compute it up to the point where our free parameters appear as was done in the case of $\tilde{\mu}$ in the one species case.
\subsection{Dimensionless form}
We start with the non-dimensionless form of the BGK model \eqref{BGK5}. We state the result for the dimensionless version performed in \cite{Pirner4}. 
One obtains  
\begin{align}
\begin{split}
\partial_t f_1 + v \cdot \nabla_x  f_1 = \frac{1}{\varepsilon_1}  (M_1 - f_1 ) + \frac{1}{\tilde{\varepsilon}_1} ( M_{12} - f_1),
\\
\partial_t f_2 + v \cdot \nabla_x  f_2 = \frac{1}{\varepsilon_2} (M_2 - f_2 ) + \frac{1}{\tilde{\varepsilon}_2} ( M_{21} - f_2),
\end{split}
\label{BGK6}
\end{align}
where 
$$ \frac{1}{\varepsilon_1} = \beta_1 \bar{\nu}_{12} \bar{t} \frac{N}{\bar{x}}$$
$$ \frac{1}{\tilde{\varepsilon}_1}= \frac{1}{\varepsilon_1} \frac{1}{\beta_1} \frac{n_2}{n_1}, $$
$$ \frac{1}{\tilde{\varepsilon}_2}= \frac{1}{\varepsilon_1} \frac{1}{\beta_1} \frac{1}{\varepsilon},  $$
$$ \frac{1}{\varepsilon_2}= \frac{1}{\varepsilon_1} \frac{\beta_2} {\beta_1 \varepsilon} \frac{n_2}{n_1}, $$
if we assume \eqref{coll5}. The parameter $\bar{\nu}_{12}$ denotes a typical value for the collision frequency, $\bar{t}$ a typical time scale, $\bar{x}$ a typical length scale and $N$ the typical number of particles in $\bar{x}^3$.
For simplicity we choose $\nu_{12}=1$ and consider the case where $\varepsilon_1$ is a small parameter. This means both type of interactions, interactions of a species with itself and interactions with the other species, become dominant and we expect to get the global equilibrium with equal mean velocities and temperatures in the limit $\varepsilon_1 \rightarrow 0$.
The non-dimensionalized Maxwell distributions are given by 
\begin{align} 
\begin{split}
M_1(x,v,t) &= \frac{n_1}{\sqrt{2 \pi T_1}^3} \exp \left({- \frac{|v-u_1|^2}{2 T_1}} \right),
\\
M_2(x,v,t) &= \frac{n_2}{\sqrt{2 \pi T_2}^3} \left( \frac{m_2}{m_1} \right) ^{\frac{1}{2}} \exp \left({- \frac{|v-u_2|^2}{2 T_2}} \frac{m_2}{m_1} \right),
\\
M_{12}(x,v,t) &= \frac{n_{1}}{\sqrt{2 \pi T_{12}}^3} \exp \left({- \frac{|v-u_{12}|^2}{2 T_{12}}}\right),
\\
M_{21}(x,v,t) &= \frac{n_{2}}{\sqrt{2 \pi T_{21}}^3} \left( \frac{m_2}{m_1} \right) ^{\frac{1}{2}} \exp \left({- \frac{|v-u_{21}|^2}{2 T_{21}}} \frac{m_2}{m_1} \right),
\end{split}
\label{Max_dimChap}
\end{align}
with the non-dimensionalized macroscopic quantities 
\begin{align}
u_{12}&=  \delta u_1 +(1- \delta) u_2, \label{convexveldim}   \\
T_{12} &= \alpha T_1 +(1-\alpha) T_2+  \frac{\gamma}{m_1} |u_1 - u_2|^2, \label{contempdim}\\
u_{21}&= ( 1- \frac{m_1}{m_2}\varepsilon(1-\delta)) u_2 + \frac{m_1}{m_2} \varepsilon (1- \delta) u_1, \label{velocdim} \\
\begin{split}
T_{21} &= [(1- \varepsilon(1- \alpha)) T_2 + \varepsilon (1- \alpha) T_1 ]\\& + \left( \frac{1}{3}\varepsilon  (1- \delta) ( \frac{m_1}{m_2} \varepsilon (\delta -1) + \delta +1) - \varepsilon \frac{\gamma}{m_1}\right)|u_1 -u_2|^2, \label{tempdim}
\end{split}
\end{align}
The macroscopic quantities in non-dimensionalized form are given by
\begin{align}
\begin{split}
&\int f_k dv = n_k,  \\ &\int v f_k dv = n_k u_k, \quad k=1,2, \\ \frac{1}{3} \frac{1}{n_1} &\int |v-u_1|^2 f_1 dv = T_1, \\ \frac{1}{3} \frac{m_2}{m_1} \frac{1}{n_2} &\int |v-u_2|^2 f_2 dv = T_2.
\end{split}
\label{momentsdim}
\end{align}
\subsection{Expansion}
Now, we want to do the analogous Chapman Enskog expansion  in the two species case. 
We expand both $f_1$ and $f_2$ in terms of $\varepsilon_1$
\begin{align*}
f_1 = f_1^0 + \varepsilon_1 f_1^1 + \varepsilon_1^2 f_1^2 + \cdots, \\
f_2 = f_2^0 + \varepsilon_1 f_2^1 + \varepsilon_1^2 f_2^2 + \cdots.
\end{align*}
From \eqref{BGK6} we get in the limit $\varepsilon_1 \rightarrow 0$ 
\begin{align*}
f_1 = \frac{1}{1+ \frac{1}{\beta_1}} (M_1 + \frac{1}{\beta_1} M_{12} ), \\
f_2 = \frac{1}{1+ \frac{1}{\beta_2}} (M_2 + \frac{1}{\beta_2} M_{21} ),
\end{align*}
from which we can deduce that both distribution functions are Maxwell distributions with equal mean velocity and temperature (see theorem 2.8 in \cite{Pirner}). \\
From \eqref{BGK6} we get
\begin{align}
\begin{split}
f_1 &= \frac{1}{\frac{1}{\varepsilon_1}+ \frac{1}{\tilde{\varepsilon}_1}}(\frac{1}{\varepsilon_1}M_1 + \frac{1}{\tilde{\varepsilon}_1} M_{12}) - \frac{1}{\frac{1}{\varepsilon_1}+ \frac{1}{\tilde{\varepsilon}_1}} (\partial_t f_1 + v \cdot \nabla_x  f_1 ), \\
f_2 &= \frac{1}{\frac{1}{\varepsilon_2}+ \frac{1}{\tilde{\varepsilon}_2}}(\frac{1}{\varepsilon_2}M_2 + \frac{1}{\tilde{\varepsilon}_2} M_{21}) - \frac{1}{\frac{1}{\varepsilon_2}+ \frac{1}{\tilde{\varepsilon}_2}} (\partial_t f_2 + v \cdot \nabla_x  f_2 ).
\end{split}
\label{BGK_um}
\end{align}
The zeroth order terms of the expansion have the same number density as the distribution functions itself so the number densities are independent of $\varepsilon_1$. But the velocities and temperatures do not coincide, since we expect a common value in equilibrium, so they depend on $\varepsilon_1$. This means the first term in the expansion $\frac{1}{\frac{1}{\varepsilon_1}+ \frac{1}{\tilde{\varepsilon}_1}}(\frac{1}{\varepsilon_1}M_1 + \frac{1}{\tilde{\varepsilon}_1} M_{12}) $ is not the zeroth order in $\varepsilon_1$ but also contains higher orders. In the one species case one directly obtains that the zeroth order is given by $M(f)$. This helps a lot in the expansion since is  possible to insert $M(f)$ as zeroth order in the expansion of the distribution function. Since we are not able to explicitly specify the zeroth order here, we are not able to do this. But we can find a linear combination of $f_1$ and $f_2$ such that the mean velocity and the temperature coincide with the zeroth order of $\frac{1}{\frac{1}{\varepsilon_1}+ \frac{1}{\tilde{\varepsilon}_1}}(\frac{1}{\varepsilon_1}M_1 + \frac{1}{\tilde{\varepsilon}_1} M_{12})$ which remembers on  the one species case. This is done in the next section.
\subsection{Combination of the distribution functions whose macroscopic quantities correspond to the macroscopic quantities of the zeroth order }
\label{sec7.5.3}
We now want to find a linear combination of the distribution functions whose mean velocity and temperature are independent of $\varepsilon_1$: \\
The zeroth order of $f_1$ is contained in
\begin{align*}
\frac{1}{\frac{1}{\varepsilon_1}+ \frac{1}{\tilde{\varepsilon}_1}}(\frac{1}{\varepsilon_1}M_1 + \frac{1}{\tilde{\varepsilon}_1} M_{12}),
\end{align*} and the zeroth order of $f_2$ is contained in 
\begin{align*}
\frac{1}{\frac{1}{\varepsilon_2}+ \frac{1}{\tilde{\varepsilon}_2}}(\frac{1}{\varepsilon_2}M_2 + \frac{1}{\tilde{\varepsilon}_2} M_{21}).
\end{align*}
So a combination of the distribution functions whose mean velocity and temperature are of zeroth order is obtained if the mean velocity and the temperature of
$
A f_1 + B f_2 $ are equal to the mean velocity and temperature of $\frac{A}{\frac{1}{\varepsilon_1}+ \frac{1}{\tilde{\varepsilon}_1}}(\frac{1}{\varepsilon_1}M_1 + \frac{1}{\tilde{\varepsilon}_1} M_{12}) + \frac{B}{\frac{1}{\varepsilon_2}+ \frac{1}{\tilde{\varepsilon}_2}}(\frac{1}{\varepsilon_2}M_2 + \frac{1}{\tilde{\varepsilon}_2} M_{21})$.
By  taking moments we get conditions on $A$ and $B$. From the velocities we get
{\small
\begin{align*}
&A n_1 u_1 + B n_2 u_2 \\ &= \frac{A}{\frac{1}{\varepsilon_1}+ \frac{1}{\tilde{\varepsilon}_1}} \left(\frac{1}{\varepsilon_1}n_1 u_1 + \frac{1}{\tilde{\varepsilon}_1} n_1 u_{12} \right) + \frac{B}{\frac{1}{\varepsilon_2}+ \frac{1}{\tilde{\varepsilon}_2}} \left(\frac{1}{\varepsilon_2} n_2 u_2 + \frac{1}{\tilde{\varepsilon}_2} n_2 u_{21} \right)\\&= \frac{A}{\frac{1}{\varepsilon_1}+ \frac{1}{\tilde{\varepsilon}_1}} \left(\frac{1}{\varepsilon_1}n_1 u_1 + \frac{1}{\tilde{\varepsilon}_1} n_1 \delta u_1 +\frac{1}{\tilde{\varepsilon}_1} n_1 (1-\delta) u_2 \right) \\&+ \frac{B}{\frac{1}{\varepsilon_2}+ \frac{1}{\tilde{\varepsilon}_2}} \left(\frac{1}{\varepsilon_2} n_2 u_2 + \frac{1}{\tilde{\varepsilon}_2} n_2 u_2 -  \frac{1}{\tilde{\varepsilon}_2} n_2 \frac{m_1}{m_2} \varepsilon ( 1- \delta) (u_2 - u_1) \right) \\&=\left[ \frac{A}{\frac{1}{\varepsilon_1}+ \frac{1}{\tilde{\varepsilon}_1}} \left(\frac{1}{\varepsilon_1} + \frac{1}{\tilde{\varepsilon}_1} \delta \right) n_1 + \frac{B}{\frac{1}{\varepsilon_2}+ \frac{1}{\tilde{\varepsilon}_2}} \frac{1}{\tilde{\varepsilon}_2} n_2  \frac{m_1}{m_2} \varepsilon ( 1- \delta) \right]  u_1\\& + \left[ \frac{A}{\frac{1}{\varepsilon_1}+ \frac{1}{\tilde{\varepsilon}_1}} \left(\frac{1}{\tilde{\varepsilon}_1}(1- \delta) n_1 + \frac{B}{\frac{1}{\varepsilon_2}+ \frac{1}{\tilde{\varepsilon}_2}} \left( \left( \frac{1}{\varepsilon_2} +\frac{1}{\tilde{\varepsilon}_2} \right) n_2 - \frac{1}{\tilde{\varepsilon}_2}  n_2\right) \frac{m_1}{m_2} \varepsilon ( 1- \delta)\right) \right]  u_2.
\end{align*}}
A comparison of the coefficient in front of $u_1$ and $u_2$ leads to
\begin{align*}
A= B \frac{\tilde{\varepsilon}_1}{\tilde{\varepsilon}_2} \frac{m_1}{m_2} \varepsilon \frac{n_2}{n_1} \frac{\frac{1}{\varepsilon_1}+ \frac{1}{\tilde{\varepsilon}_1}}{\frac{1}{\varepsilon_2}+ \frac{1}{\tilde{\varepsilon}_2}}.
\end{align*}
If we do the same with the temperatures, we will obtain the same result. \\
We can simplify this expression by using the relationships between $\varepsilon_1$, $\varepsilon_2$, $\tilde{\varepsilon}_1$ and $\tilde{\varepsilon}_2$. First we get
\begin{align*}
\frac{\tilde{\varepsilon}_1}{\tilde{\varepsilon}_2} = \frac{1}{\varepsilon^2} \frac{n_1}{n_2},
\end{align*} and secondly
\begin{align*}
\frac{\frac{1}{\varepsilon_1}+ \frac{1}{\tilde{\varepsilon}_1}}{\frac{1}{\varepsilon_2}+ \frac{1}{\tilde{\varepsilon}_2}} = \frac{\beta_1 \frac{n_1}{n_2} + 1}{\beta_2 \frac{n_2}{n_1} +1} \varepsilon^2.
\end{align*}
Multiplying the two expressions we get
\begin{align*}
\frac{\beta_1 \frac{n_1}{n_2} +1}{\beta_2 \frac{n_2}{n_1} +1} \frac{n_1}{n_2}.
\end{align*}
So all in all we get the following simplification
\begin{align*}
A= B \frac{m_1}{m_2} \varepsilon \frac{\beta_1 \frac{n_1}{n_2} + 1}{\beta_2 \frac{n_2}{n_1} + 1}.
\end{align*}
We choose $B= 1$ and $ A= \frac{m_1}{m_2} \varepsilon \frac{\beta_1 \frac{n_1}{n_2} + 1}{\beta_2 \frac{n_2}{n_1}+ 1}$ and obtain
\begin{align}
\frac{m_1}{m_2} \varepsilon \frac{\beta_1 \frac{n_1}{n_2}+ 1}{\beta_2 \frac{n_2}{n_1} + 1} f_1 +  f_2.
\label{combination}
\end{align}
So 
\begin{align}
\frac{m_1}{m_2} \varepsilon \frac{\beta_1 \frac{n_1}{n_2} + 1}{\beta_2 \frac{n_2}{n_1} + 1} f_1 +  f_2 = \bar{f}^0 + \varepsilon_1 \bar{f}^1 + \cdots.
\label{same_moments}
\end{align}
where $\bar{f}^0$ has the same density, velocity and temperature according to \eqref{momentsdim} as the left hand-side and moments of $\bar{f}^k$ $, k\geq 1$ are zero as in the one species case. We can explicitly compute the moments of $\bar{f}_0$. This is done in the next section.
\subsection{Moments of $\bar{f}^0$}
The density of $\bar{f}^0$ is given by
\begin{align*}
\bar{n}^0 = \int \bar{f}^0 dv = \int (\frac{m_1}{m_2} \varepsilon \frac{\beta_1 \frac{n_1}{n_2} + 1}{\beta_2 \frac{n_2}{n_1} + 1} f_1 +  f_2) dv = \frac{m_1}{m_2} \varepsilon \frac{\beta_1 \frac{n_1}{n_2} + 1}{\beta_2 \frac{n_2}{n_1}+ 1} n_1 + n_2.
\end{align*}
Therefore the velocity is given by
\begin{align*}
\bar{u}^0 = \frac{1}{\bar{n}^0} \int \bar{f}^0 v dv = \frac{\frac{m_1}{m_2} \varepsilon \frac{\beta_1 + 1}{\beta_2 \frac{n_2}{n_1} + 1} n_1 u_1 + n_2 u_2}{\frac{m_1}{m_2} \varepsilon \frac{\beta_1 \frac{n_1}{n_2}+ 1}{\beta_2 + 1} n_1 +  n_2 },
\end{align*}
and the energy is given by
\begin{align*}
\bar{n}^0 |\bar{u}^0|^2 + 3 \bar{n}^0 \frac{\bar{T}^0}{\bar{m}^0} = \int \bar{f}^0 |v|^2 dv \\= A n_1 |u_1|^2 +  n_2 |u_2|^2 + 3 n_1 A T_1 + 3 n_2  T_2 \frac{m_1}{m_2}.
\end{align*}
Solving this for $\frac{\bar{T}^0}{\bar{m}^0}$ leads to the temperature
\begin{align*}
\frac{\bar{T}^0}{\bar{m}^0} = \frac{1}{3} \frac{A n_1 n_2}{(An_1+  n_2)^2} |u_1-u_2|^2 + \frac{A n_1}{A n_1 +  n_2 } T_1 + \frac{ n_2}{A n_1 +  n_2} T_2 \frac{m_1}{m_2}.
\end{align*}

\subsection{Combining the distribution functions whose moments are zero}

We observe that in section \ref{sec7.5.3} we regained a property of the one species case. But since we are in the two species case it is not enough to have one equation for the sum of the two distribution functions. We need a second equation. 
Since we expect that in equilibrium the mean velocities and the temperatures of the two distribution functions are the same we know that the zeroth order of $\frac{n_2}{n_1} f_1 - f_2$ has zero mean velocity and the zeroth order of the combination $\frac{n_2}{n_1} \frac{m_1}{m_2} f_1 - f_2$ has zero temperature. \\
Therefore
\begin{align}
\frac{n_2}{n_1} f_1 -  f_2 = \tilde{f}^0 + \varepsilon_1 \tilde{f}^1 + \cdots,
\label{zero_moments} \\
\frac{n_2}{n_1}\frac{m_1}{m_2} f_1 -  f_2 = \tilde{\tilde{f}}^0 + \varepsilon_1 \tilde{\tilde{f}}^1 + \cdots,
\label{zero_temp}
\end{align}
where $\tilde{f}^0$ has zero mean velocity and $\tilde{\tilde{f}}^0$ has zero temperature. \\ 
Solving \eqref{same_moments} and \eqref{zero_moments} for $f_1$ and $f_2$ leads to
\begin{align}
\begin{split}
f_1 = \frac{1}{A+ \frac{n_2}{n_1}} \left(\bar{f}^0 + \tilde{f}^0 \right)&+ \varepsilon_1 \frac{1}{A+ \frac{n_2}{n_1}} \left( \bar{f}^1 + \tilde{f}^1 \right) \\& +\varepsilon_1^2 \frac{1}{A+\frac{n_2}{n_1}}\left(\bar{f}^2 + \tilde{f}^2 \right)  + O(\varepsilon_1^3),
\end{split}
\label{expans1}
\end{align}
\begin{align}
\begin{split}
f_2 = \frac{1}{\frac{n_2}{n_1} + A} ( \frac{n_2}{n_1} \bar{f}^0 -A \tilde{f	}^0 ) &+ \varepsilon_1 \frac{1}{\frac{n_2}{n_1} + A} \left(\frac{n_2}{n_1}  \bar{f}^1 - A \tilde{f}^1 \right) \\&+ \varepsilon^2_1 \frac{1}{\frac{n_2}{n_1}+A} \left(\frac{n_2}{n_1} \bar{f}^2 - A \tilde{f}^2 \right) + O(\varepsilon_1^3),
\end{split}
\label{expans2}
\end{align}
and solving \eqref{same_moments} and \eqref{zero_temp} for $f_1$ and $f_2$ leads to
\begin{align}
\begin{split}
f_1 = \frac{1}{A+ \frac{n_2}{n_1}\frac{m_1}{m_2}} \left(\bar{f}^0 + \tilde{\tilde{f}}^0 \right)+ \varepsilon_1 \frac{1}{A+ \frac{n_2}{n_1}\frac{m_1}{m_2}} \left( \bar{f}^1 + \tilde{\tilde{f}}^1 \right) + O(\varepsilon_1^2),
\end{split}
\label{expans3}
\end{align}
\begin{align}
\begin{split}
f_2 = \frac{1}{\frac{n_2}{n_1}\frac{m_1}{m_2} + A} \left( \frac{n_2}{n_1}\frac{m_1}{m_2} \bar{f}^0 -A \tilde{\tilde{f}}^0 \right) + \varepsilon_1 \frac{1}{\frac{n_2}{n_1}\frac{m_1}{m_2} + A} \left(\frac{n_2}{n_1}\frac{m_1}{m_2}  \bar{f}^1 - A \tilde{\tilde{f}}^1 \right)\\+ \varepsilon^2_1 \frac{1}{\frac{n_2}{n_1}\frac{m_1}{m_2} + A} \left(\frac{n_2}{n_1}\frac{m_1}{m_2}  \bar{f}^2 - A \tilde{\tilde{f}}^2 \right) + O(\varepsilon_1^2).
\end{split}
\label{expans4}
\end{align}
So for $\varepsilon_1 \rightarrow 0$ we see from \eqref{expans1} and \eqref{expans2} that the zeroths order are two Maxwell distributions $f^M_1$ and $f^M_2$ with the common velocity 
\begin{align*}
\bar{u}:=\frac{1}{A n_1 + n_2}(A n_1 u_1 + n_2 u_2),
\end{align*}
and from \eqref{expans3} and \eqref{expans4} that the zeroths order are two Maxwell distributions $f^M_1$ and $f^M_2$ with the equal temperature 
\begin{align*}
\bar{T}:= \frac{1}{A n_1 + n_2 \frac{m_1}{m_2}} (A n_1 + n_2) \frac{\bar{T}^0}{\bar{m}^0}.
\end{align*}
Remember from the remark below \eqref{BGK_um} that $f_1^M$ has density $n_1$ and $f_2^M$ has density $n_2$. \\
Using this we can deduce from \eqref{BGK_um} by inserting the expansions of $f_1$ and $f_2$
\begin{align}
\begin{split}
f_1 = \frac{1}{\frac{1}{\varepsilon_1}+ \frac{1}{\tilde{\varepsilon}_1}}(\frac{1}{\varepsilon_1}M_1 + \frac{1}{\tilde{\varepsilon}_1} M_{12}) - \frac{1}{\frac{1}{\varepsilon_1}+ \frac{1}{\tilde{\varepsilon}_1}} \int (\partial_t f_1^M + v \cdot \nabla_x  f_1^M) ) dv +O(\varepsilon_1^2), \\
f_2 = \frac{1}{\frac{1}{\varepsilon_2}+ \frac{1}{\tilde{\varepsilon}_2}}(\frac{1}{\varepsilon_2}M_2 + \frac{1}{\tilde{\varepsilon}_2} M_{21}) - \frac{1}{\frac{1}{\varepsilon_2}+ \frac{1}{\tilde{\varepsilon}_2}} \int (\partial_t f_2^M + v \cdot \nabla_x  f_2^M) dv+ O(\varepsilon_1^2).
\end{split}
\label{BGKexp}
\end{align}
\subsection{Expansion of the velocity, the pressure tensor and the heat flux}
 The exact macroscopic conservation equations that need to be closed in the two species  read as
 {\scriptsize
 \begin{multline*}
\\
\partial_t n_1 + \nabla_x \cdot(n_1 u_1)=0,
\\
\partial_t n_2 + \nabla_x \cdot(n_2 u_2)=0,
\\
 \partial_t( n_1 u_1)+\nabla_x \cdot \mathbb{P}_1 + \nabla_x \cdot ( n_1 u_1 \otimes u_1  ) = 
\frac{1}{\tilde{\varepsilon}_1} (1 - \delta) (u_2 - u_1) ,
\\
 \partial_t( n_2 u_2)+\nabla_x \cdot \mathbb{P}_2 \frac{m_1}{m_2}+ \nabla_x \cdot ( n_2 u_2 \otimes u_2  )  = 
\frac{1}{\tilde{\varepsilon}_1}  (1-\delta) (u_1-u_2),
\\
\partial_t \left(\frac{1}{2} n_1 |u_1|^2 + \frac{3}{2} n_1 T_1 \right) + \nabla_x \cdot Q_1  \\=  \frac{1}{2} \frac{1}{\tilde{\varepsilon}_1}((\delta^2-1) |u_1|^2 + (1- \delta)^2 |u_2|^2 + 2 \delta ( 1- \delta) u_1 \cdot u_2) + \frac{3}{2} \frac{1}{\tilde{\varepsilon}_1} ((1-\alpha) (T_2-T_1) + \frac{\gamma}{m_1} |u_1-u_2|^2),
\\
\partial_t \left(\frac{1}{2} n_2 |u_2|^2 + \frac{3}{2} n_2 T_2 \frac{m_1}{m_2} \right) + \nabla_x \cdot Q_2  \\=-\frac{1}{2} \frac{1}{\tilde{\varepsilon}_1}((\delta^2-1) |u_1|^2 + (1- \delta)^2 |u_2|^2 + 2 \delta ( 1- \delta) u_1 \cdot u_2) - \frac{3}{2} \frac{1}{\tilde{\varepsilon}_1} ((1-\alpha) (T_2-T_1) + \frac{\gamma}{m_1} |u_1-u_2|^2),
\\
\end{multline*}}
This is obtained by computing moments (multiplying \eqref{BGK6} by $(1,v,|v|^2)$, integrating with respect to $v$ and using the definitions of the macroscopic quantities in \eqref{momentsdim}),
where we need expressions for the species velocities, pressure tensors and energy fluxes
\begin{align}
 &u_k(x,t)=\frac{1}{n_k(x,t)}\int v f_k(x,v,t) dv, \label{expvel}\\
 &\mathbb{P}_k(x,t)= \int (v-u_k(x,t))\otimes (v-u_k(x,t))f_k(x,v,t) dv, \\
 &Q_k(x,t) = \frac{1}{2} \int |v|^2 v f_k(x,v,t) dv, \label{expheat}
\end{align}
for $k=1,2$. In the following we want to insert the expansions for $f_1^{\varepsilon_1}$ and $f_2^{\varepsilon_2}$ into these three integrals in order to see if we are able to fix the diffusion coefficient, the viscosity, the heat conductivity and the thermal diffusion parameter as described in section \ref{sec7.3.6.1}.

\subsubsection{Expansion of the velocities}
First consider the velocity. If we just insert \eqref{BGKexp} into \eqref{expvel}, this will lead to expansions of the form
\begin{align}
\begin{split}
 u_1^{\varepsilon_1} = u_2^{\varepsilon_1} + O(\varepsilon_1), \\
 u_2^{\varepsilon_1} = u_1^{\varepsilon_1} + O(\varepsilon_1).
 \end{split}
 \label{expvel1}
 \end{align}
 This is in accordance to our expectation that for $\varepsilon_1 \rightarrow 0$ the two velocities converge to a common value, but the expansion \eqref{expvel1} cannot be used to solve it for the two velocities $u_1^{\varepsilon_1}$ and $u_2^{\varepsilon_1}$. In order to do this we need an additional step. This is done as follows.
 The velocity of the ion expansion of the first term in \eqref{BGKexp} using the expression of $u_{12}$ is given by
\begin{align*}
\frac{1+\frac{1}{\beta_1} \delta}{1+\frac{1}{\beta_1}} u_1^{\varepsilon_1} + \frac{\frac{1}{\beta_1} (1- \delta)}{1+\frac{1}{\beta_1}} u_2^{\varepsilon_1}.
\end{align*}
We can split this expression into
\begin{align*}
\bar{u} + \frac{-A n_1 (1- \delta) + n_2 ( \beta_1 + \delta)}{(\beta_1 + 1) ( A n_1 + n_2)} u_1^{\varepsilon_1} + \frac{A n_1 ( 1 - \delta) - n_2 ( \beta_1 + \delta)}{(\beta_1 + 1) (An_1 + n_2)} u_2^{\varepsilon_1}.
\end{align*}
We denote 
\begin{align*}
c_1 :=  \frac{-A n_1 (1- \delta) + n_2 ( \beta_1 + \delta)}{(\beta_1 + 1) ( A n_1 + n_2)},
\end{align*}
so
\begin{align*}
\frac{1+\frac{1}{\beta_1} \delta}{1+\frac{1}{\beta_1}} u_1^{\varepsilon_1} + \frac{\frac{1}{\beta_1} (1- \delta)}{1+\frac{1}{\beta_1}} u_2^{\varepsilon_1}=\bar{u} + c_1 u_1^{\varepsilon_1}- c_1 u_2^{\varepsilon_1}.
\end{align*}
So we see from \eqref{BGKexp} that we get 
\begin{align*}
u_1^{\varepsilon_1} = \bar{u} + c_1 u_1^{\varepsilon_1} - c_1 u_2^{\varepsilon_1} - \frac{1}{\frac{1}{\varepsilon_1} + \frac{1}{\tilde{\varepsilon}_1}} \frac{1}{n_1} \int v \left( \partial_t f_1^M + \nabla_x \cdot (v f_1^M) \right) dv \\ + O(\varepsilon_1^2) .
\end{align*}
Solving this for $u_1^{\varepsilon_1}$ leads to
\begin{align}
\begin{split}
u_1^{\varepsilon_1} &= - \frac{c_1}{1-c_1} u_2^{\varepsilon_1} +\frac{1}{1-c_1} \bar{u}  \\-& \frac{1}{1-c_1} \frac{1}{\frac{1}{\varepsilon_1} + \frac{1}{\tilde{\varepsilon}_1}} \frac{1}{n_1} \int v \left( \partial_t f_1^M + \nabla_x \cdot (v f_1^M) \right) dv +\frac{1}{1-c_1} O(\varepsilon_1^2).
\end{split}
\label{expr1}
\end{align}
Similarly, we get for the second species
\begin{align}
\begin{split}
u_2^{\varepsilon_1}&=-\frac{A \frac{n_1}{n_2} c_1}{1- A \frac{n_1}{n_2} c_1 } u_1^{\varepsilon_1} + \frac{1}{1- A \frac{n_1}{n_2} c_1} \bar{u} \\-& \frac{1}{1-A \frac{n_1}{n_2} c_1} \frac{1}{\frac{1}{\varepsilon_2} + \frac{1}{\tilde{\varepsilon}_2}} \frac{1}{n_2} \int v \left( \partial_t f_2^M + \nabla_x \cdot (v f_2^M)\right) dv+ \frac{1}{1- A \frac{n_1}{n_2} c_1} O(\varepsilon_1^2).
\end{split}
\label{expr2}
\end{align}
Solving \eqref{expr1} and \eqref{expr2} for $u_1^{\varepsilon_1}$ and $u_2^{\varepsilon}$ we get
\begin{align}
\begin{split}
u_1^{\varepsilon_1} =\bar{u}- \frac{1- A \frac{n_1}{n_2}c_1}{1-A \frac{n_1}{n_2} c_1 -c_1}\frac{1}{\frac{1}{\varepsilon_1} + \frac{1}{\tilde{\varepsilon}_1}} \frac{1}{n_1} \int v \left( \partial_t f_1^M + \nabla_x \cdot (v f_1^M) \right) dv \\+ \frac{c_1}{1-A \frac{n_1}{n_2}c_1 - c_1}\frac{1}{\frac{1}{\varepsilon_2} + \frac{1}{\tilde{\varepsilon}_2}} \frac{1}{n_2} \int v \left( \partial_t f_2^M + \nabla_x \cdot (v f_2^M) \right) dv+ O(\varepsilon_1^2), \\
u_2^{\varepsilon_1} = \bar{u} - \frac{A \frac{n_1}{n_2} c_1}{-1+ c_1 + A \frac{n_1}{n_2} c_1}\frac{1}{\frac{1}{\varepsilon_1} + \frac{1}{\tilde{\varepsilon}_1}} \frac{1}{n_1} \int v \left( \partial_t f_1^M + \nabla_x \cdot (v f_1^M)\right) dv \\+ \frac{1-c_1}{-1 + c_1 + A \frac{n_1}{n_2}c_1}\frac{1}{\frac{1}{\varepsilon_2} + \frac{1}{\tilde{\varepsilon}_2}} \frac{1}{n_2} \int v \left( \partial_t f_2^M + \nabla_x \cdot (v f_2^M) \right) dv+ O(\varepsilon_1^2). \end{split}
\label{expvelint}
\end{align}

\begin{remark}
At this point we observe that we cannot use the parameter $\gamma$ from \eqref{contempdim} because $|u_1^{\varepsilon_1} - u_2^{\varepsilon_1}|^2$ is of order $\varepsilon_1^2$ and has no influence on the order $\varepsilon_1$.
 \end{remark}
 \begin{remark}
 According to section \ref{sec7.3.6.1} we  have a diffusion coefficient in front of the integrals in \eqref{expvelint}. In our expansion in \eqref{expvelint} this coefficient in front of the integrals depends on the free parameter $\delta$, since $c_1$ depends on $\delta$. This enables us to determine the free parameter $\delta$ later such that the coefficient in front of the integrals fits to the diffusion coefficient measured by experiments.
 \end{remark}
\subsubsection{Expansion of the temperature}
For the temperature we do the same trick as in the case of the velocities.
$\frac{m_1}{3 n_1}$ times the temperature of the ion expansion of the first term in \eqref{BGKexp} using the expression of $T_{12}$ is given by
\begin{align*}
\frac{1+\frac{1}{\beta_1} \alpha}{1+\frac{1}{\beta_1}}  T_1^{\varepsilon_1} + \frac{\frac{1}{\beta_1} (1- \alpha)}{1+\frac{1}{\beta_1}} T_2^{\varepsilon_1}.
\end{align*}
We can split this expression into
\begin{align*}
\bar{T} + \frac{-A n_1 (1- \alpha) + \frac{m_1}{m_2}n_2 (\beta_1 + \alpha)}{(\beta_1 +1)(An_1 + \frac{m_1}{m_2}n_2)} T_1^{\varepsilon_1} + \frac{A n_1 (1- \alpha) - \frac{m_1}{m_2}n_2 (\beta_1 + \alpha)}{(\beta_1 +1)(An_1 + \frac{m_1}{m_2}n_2)} T_2^{\varepsilon_1}.
\end{align*}
We denote 
\begin{align*}
c_2:=\frac{-A n_1 (1- \alpha) + \frac{m_1}{m_2}n_2 (\beta_1 + \alpha)}{(\beta_1 +1)(An_1 + \frac{m_1}{m_2}n_2)},
\end{align*}
so
\begin{align*}
\frac{1+\frac{1}{\beta_1} \alpha}{1+\frac{1}{\beta_1}}  T_1^{\varepsilon_1} + \frac{\frac{1}{\beta_1} (1- \alpha)}{1+\frac{1}{\beta_1}} T_2^{\varepsilon_1} = \bar{T} +c_2 T_1^{\varepsilon_1} - c_2 T_2^{\varepsilon_1}.
\end{align*}
We see from \eqref{BGKexp} 
\begin{align*}
\begin{split}
T_1^{\varepsilon_1} =  \bar{T} +c_2 T_1^{\varepsilon_1} - c_2 T_2^{\varepsilon_1} \\ - \frac{1}{\frac{1}{\varepsilon_1}+ \frac{1}{\tilde{\varepsilon}_1}} \frac{1}{3n_1}\int |v- \bar{u}|^2 (\partial_t f_1^M) + \nabla_x \cdot (v f_1^M) )dv+O(\varepsilon_1^2).
\end{split}
\end{align*}
Solving this for $T_1^{\varepsilon_1}$ leads to
\begin{align}
\begin{split}
T_1^{\varepsilon_1} = \frac{1}{1-c_2} \bar{T}  - \frac{c_2}{1-c_2} T_2^{\varepsilon_1} \\ - \frac{1}{\frac{1}{\varepsilon_1}+ \frac{1}{\tilde{\varepsilon}_1}} \frac{1}{3n_1} \frac{1}{1-c_2} \int |v- \bar{u}|^2 (\partial_t f_1^M + \nabla_x \cdot (v f_1^M) )dv \\ + \frac{1}{1-c_1}(\varepsilon_1^2).
\end{split}
\label{expr3}
\end{align}
Similarly, we get for the second species
\begin{align}
\begin{split}
T_2^{\varepsilon_1} = \frac{1}{1- A \frac{n_1}{n_2} \frac{m_2}{m_1} c_2} \bar{T}  - \frac{A \frac{n_1}{n_2} \frac{m_2}{m_1}c_2}{1- A \frac{n_1}{n_2} \frac{m_2}{m_1} c_2} T_1^{\varepsilon_1} \\ - \frac{1}{\frac{1}{\varepsilon_2}+ \frac{1}{\tilde{\varepsilon}_2}} \frac{m_2}{m_1} \frac{1}{3n_2}\frac{1}{1- A \frac{n_1}{n_2} \frac{m_2}{m_1} c_2}  \int |v- \bar{u}|^2 (\partial_t f_2^M) + \nabla_x \cdot (v f_2^M) )dv \\ + \frac{1}{1- A \frac{n_1}{n_2} \frac{m_2}{m_1}c_2} O(\varepsilon_1^2).
\end{split}
\label{expr4}
\end{align}
Solving \eqref{expr3} and \eqref{expr4} for $T_1^{\varepsilon_1}$ and $T_2^{\varepsilon}$ we get
{\small
\begin{align}
\begin{split}
T_1^{\varepsilon_1} = \bar{T}  -  \frac{1}{\frac{1}{\varepsilon_1}+ \frac{1}{\tilde{\varepsilon}_1}} \frac{1- A \frac{n_1}{n_2}\frac{m_2}{m_1} c_2}{1- A \frac{n_1}{n_2} \frac{m_2}{m_1}c_2 - c_2} \frac{1}{3 n_1}  \int |v- \bar{u}|^2 (\partial_t f_1^M + \nabla_x \cdot (v f_1^M) )dv \\ - \frac{1}{\frac{1}{\varepsilon_2}+ \frac{1}{\tilde{\varepsilon}_2}} \frac{m_2}{m_1}\frac{1}{3n_2} \frac{c_2}{1- A \frac{n_1}{n_2} \frac{m_2}{m_1} c_2 - c_2} \int |v- \bar{u}|^2 (\partial_t f_2^M) + \nabla_x \cdot (v f_2^M) )dv + O(\varepsilon_1^2),
\end{split}
\label{exptemp1}
\end{align}}
{\small
\begin{align}
\begin{split}
T_2^{\varepsilon_1} = \bar{T} &- \frac{1}{\frac{1}{\varepsilon_2}+ \frac{1}{\tilde{\varepsilon}_2}} \frac{1-c_2}{-1 + c_2 + A \frac{n_1}{n_2} \frac{m_2}{m_1} c_2} \frac{m_2}{m_1}\frac{1}{3 n_2}  \int |v- \bar{u}|^2 (\partial_t f_2^M + \nabla_x \cdot (v f_2^M) )dv \\ &- \frac{1}{\frac{1}{\varepsilon_1}+ \frac{1}{\tilde{\varepsilon}_1}} \frac{1}{3n_1} \frac{A \frac{n_1}{n_2} \frac{m_2}{m_1} c_2}{-1 + c_2 + A \frac{n_1}{n_2}\frac{m_2}{m_1} c_2} \int |v- \bar{u}|^2 (\partial_t f_1^M) + \nabla_x \cdot (v f_1^M) )dv \\&+ O(\varepsilon_1^2).
\end{split}
\label{exptemp2}
\end{align}}
If we compute the expansion of the non-diagonal elements of the pressure tensor, the zeroth order vanishes and we obtain
\begin{align}
\begin{split}
&\int (v_l - u_{1,l})(v_m - u_{1,m}) f_1 dv = \\&- \frac{1}{\frac{1}{\varepsilon_1}+ \frac{1}{\tilde{\varepsilon}_1}} \int (v_l- u_{1,l})(v_m-u_{1,m}) (\partial_t f_1^M + \nabla_x \cdot (v f_1^M) )dv + O(\varepsilon_1^2),
\\
&\int (v_l - u_{2,l})(v_m - u_{2,m}) f_2 dv = \\&- \frac{1}{\frac{1}{\varepsilon_2}+ \frac{1}{\tilde{\varepsilon}_2}} \int (v_l- u_{2,l})(v_m-u_{2,m}) (\partial_t f_2^M + \nabla_x \cdot (v f_2^M) )dv + O(\varepsilon_1^2).
\end{split}
\label{express}
\end{align}
\begin{remark}
 According to section \ref{sec7.3.6.1} we  have a viscosity coefficient in front of the integrals of the expansion of the pressure tensor in \eqref{express}. This coefficient depends on our free parameter $\alpha$, since the constant $c_2$ depends on the undetermined parameter $\alpha$ from \eqref{contempdim}. This enables us to determine the parameter $\alpha$ later such that it fits to the viscosity coefficient measured by experiments.
\end{remark}
\subsubsection{Expansion of the energy flux}
If we insert the expansion \eqref{BGKexp} into \eqref{expheat} and use the definition of the mixture Maxwell distributions \eqref{BGKmix5}, we get
{\small
\begin{align}
\begin{split}
&\frac{1}{2} \int |v|^2 v f_1 dv = \frac{5}{2} \frac{1}{1+ \beta_1 \frac{n_1}{n_2}} n_1  \\ &[ (\beta_1 \frac{n_1}{n_2}+ \alpha \delta) T_1^{\varepsilon_1} u_1^{\varepsilon_1} + \alpha (1- \delta) T_1^{\varepsilon_1} u_2^{\varepsilon_1} + (1- \alpha) \delta T_2^{\varepsilon_1} u_1^{\varepsilon_1} + (1- \alpha)(1-\delta) T_2^{\varepsilon_1} u_2^{\varepsilon_1}]\\ &+ \frac{1}{2} \frac{1}{1+ \beta_1 \frac{n_1}{n_2}} n_1 [(\beta_1 \frac{n_1}{n_2}  |u_1^{\varepsilon_1}|^2 u_1^{\varepsilon_1} +|\delta u_1^{\varepsilon_1} + (1- \delta) u_2^{\varepsilon_1}|^2 ( \delta u_1^{\varepsilon_1} + (1- \delta) u_2^{\varepsilon_1})]  \\ &-\frac{1}{2}  \frac{1}{\frac{1}{\varepsilon_1}+ \frac{1}{\tilde{\varepsilon}_1}} ( \int |v|^2 v (\partial_t f_1^M + \nabla_x \cdot (v f_1^M) )dv + O(\varepsilon_1^2),
\end{split}
\label{exheat1}
\end{align}}
in which we can insert the expansions for the velocities and the temperatures. The zeroth order is given by
$$\frac{5}{2}  n_1  (\bar{T} \bar{u} + |\bar{u}|^2 \bar{u}).$$
The energy flux for species $2$ is given by
\begin{align}
\begin{split}
&\frac{m_2}{m_1}\frac{1}{2} \int |v|^2 v f_2 dv =  \frac{5}{2} \frac{1}{1+ \beta_1 \frac{n_1}{n_2}} n_2  \\ &[ \beta_2\frac{n_2}{n_1} T_2^{\varepsilon_1} u_2^{\varepsilon_1} + ( \alpha T_1^{\varepsilon_1} + (1- \alpha)T_2^{\varepsilon_1})(u_2^{\varepsilon_1} - \frac{m_1}{m_2}\varepsilon (1- \delta) (u_2^{\varepsilon_1}-u_1^{\varepsilon_1})]\\ &+ \frac{1}{2} \frac{1}{1+ \beta_2 \frac{n_2}{n_1}} n_1 [\beta_2 \frac{n_2}{n_1}|u_2^{\varepsilon_1}|^2 u_2^{\varepsilon_1} + | u_2^{\varepsilon_1}\\& - \frac{m_1}{m_2} \varepsilon (1- \delta)(u_2^{\varepsilon_1}-u_1^{\varepsilon_1})|^2 (u_2^{\varepsilon_1} - \frac{m_1}{m_2 } \varepsilon (1- \delta) (u_2^{\varepsilon_1} -u_1^{\varepsilon_1}))] \\ &- \frac{1}{2}  \frac{1}{\frac{1}{\varepsilon_2}+ \frac{1}{\tilde{\varepsilon}_2}} ( \int |v|^2 v (\partial_t f_2^M + \nabla_x \cdot (v f_2^M) )dv + O(\varepsilon_1^2).
\end{split}
\label{exheat2}
\end{align}
Here, the zeroth order is given by
$$ \frac{5}{2}  n_2( \bar{T} \bar{u} + |\bar{u}|^2 \bar{u}).$$
\begin{remark}
 According to section \ref{sec7.3.6.1} we  have the thermal conductivity and the thermal diffusion parameter in front of the integrals in the expansion of the energy flux in \eqref{exheat1} and \eqref{exheat2}  which we want to be able to fix  the value which one obtains from experiments. We observe that we do not have more free parameters since $\alpha$ and $\delta$ are already fixed in order to obtain the right viscosity and diffusion coefficient.   But if we perform the extension of the BGK to an ES-BGK model we will obtain an additional free parameter in the pressure tensor similar as it is done in the one species case. This allows to determine the additional free parameter in order to get the right viscosity coefficient. The parameter $\alpha$ remains undetermined such that we can use $\alpha$ to fix the thermal conductivity in the heat flux. A fourth free parameter is gained if we treat the collision frequency $\nu_{12}$ as a free parameter as it is done in \cite{Groppi}. With this we can determine the diffusion coefficient such that the parameter $\delta$ remains undetermined and we can fix $\delta$ in the heat flux expansion such that the thermal diffusion parameter has the right physical value.
\end{remark}

\section{Extensions to an ES-BGK approximation}
\label{sec7.4}
In this section we want to present three possible extensions to an ES-BGK model for gas mixtures. The first one has the attempt to keep it as simple as possible and extend only the Maxwell distribution in the single relaxation term. This describes the relaxation of the distribution function to an equilibrium distribution due to interactions of the species with itself. The two other extensions try to do it more symmetrically and extend every Maxwell distribution. The first ansatz in this case is to extend it exactly analogously as in the one species case and the second ansatz proposes a different extension which is motivated by the physical intuition of the physicist Holway who invented the ES-BGK model in \cite{Holway}. These models are also presented in \cite{Pirner2} by Klingenberg, Pirner and Puppo. For the reader's convenience we want to repeat it here.

 The simplest choice is to only  replace the collision operators which represent the collisions of a species with itself by the ES-BGK collision operator for one species suggested in \cite{AndriesPerthame2001}. Then the model can be written as:

\begin{align} \begin{split} \label{ESBGKsimple}
\partial_t f_k + v \cdot \nabla_x  f_k   &= \nu_{kk} n_k (G_k - f_k ) + \nu_{kj} n_j ( M_{kj} - f_k), \quad k,j=1,2, ~ j \neq k,
\end{split}
\end{align}
with the modified Maxwell distributions
\begin{align} 
\begin{split}
G_k(x,v,t)&= \frac{n_k}{\sqrt{det(2 \pi \frac{\mathcal{T}_k}{m_k})}} \exp \left(- \frac{1}{2} (v-u_k) \cdot \left(\frac{\mathcal{T}_k}{m_k} \right)^{-1} \cdot (v-u_k)\right), 
\end{split}
\label{ESBGKmixsimple}
\end{align}
for $k=1,2$ and $M_{12}, M_{21}$ the Maxwell distributions described in the previous sections.
 $G_1$ and $G_2$ have the same densities, velocities and pressure tensors as $f_1$ and $f_2$, respectively, so we still guarantee the conservation of mass, momentum and energy in interactions of one species with itself. 
Since the first term describes the interactions of a species with itself, it should correspond to the single ES-BGK collision operator suggested in section \ref{sec7.1}. So we choose $\mathcal{T}_1$ and $\mathcal{T}_2$ as 
\begin{align}
\mathcal{T}_k= (1- \tilde{\mu}_k) T_k \textbf{1} + \tilde{\mu}_k \frac{\mathbb{P}_k}{n_k}, 
\label{ten}
\end{align}
 with $\tilde{\mu}_k \in \mathbb{R}$, $k=1,2$ being free parameters which we can choose in a way to fix physical parameters in the Navier-Stokes equations. So, all in all, together with the parameters in the mixture Maxwell distributions \eqref{convexvel5} and \eqref{contemp5} we now have five free parameters, see \eqref{convexvel5}, and \eqref{contemp5}  for the other free parameters. 

We can prove that this model is well-defined, satisfies conservation of mass, momentum and energy; that it has an H-Theorem and we can characterize the equilibrium as Maxwellians with the same mean velocity and temperature. For the proofs see \cite{Pirner2}.
In the following, we also want to replace the scalar temperatures in the mixture Maxwell distributions by a tensor.
In the first model the terms $(v_j - u_{kj})f_k (v_i - u_{ki})$ for $i\neq j$ do not appear in the relaxation operator. To obtain a more detailed description of the viscous effects in the mixture we take into account these cross terms during the relaxation process. Then the model can be written as:
\begin{align} \begin{split} \label{ESBGK5}
\partial_t f_k + v \cdot \nabla_x  f_k   &= \nu_{kk} n_k (G_k - f_k ) + \nu_{kj} n_j ( G_{kj} - f_k), \quad k=1,2, k \neq j, 
\end{split}
\end{align}
with the modified Maxwell distributions
\begin{align} 
\begin{split}
G_k(x,v,t)&= \frac{n_k}{\sqrt{\det(2 \pi \frac{\mathcal{T}_k}{m_k})}} \exp(- \frac{1}{2} (v-u_k) \cdot \left(\frac{\mathcal{T}_k}{m_k} \right)^{-1} \cdot (v-u_k)),
\\
G_{kj}(x,v,t) &= \frac{n_k}{\sqrt{\det(2 \pi \frac{\mathcal{T}_{kj}}{m_k})}} \exp(- \frac{1}{2} (v-u_{kj}) \cdot \left(\frac{\mathcal{T}_{kj}}{m_k} \right)^{-1} \cdot (v-u_{kj})),
\end{split}
\label{ESBGKmix5}
\end{align}
for $k=1,2, k \neq j$ and with $\mathcal{T}_k$ defined by \eqref{ten}.
Again, the conservation of mass, momentum and energy in interactions of one species with itself is ensured by this choice of the modified Maxwell distributions $G_1$ and $G_2$ which have the same densities, velocities and pressure tensor as $f_1$ and $f_2$, respectively. In addition, the choice of the densities in $G_{12}$ and $G_{21}$, we also guarantee conservation of mass in interactions of one species with the other one.
If we extend $T_{12}$ and $T_{21}$ in the same fashion to a tensor as in the case of one species, we obtain
{\small
\begin{align}
\mathcal{T}_{12} &=(1- \tilde{\mu}_{12}) (\alpha T_1 + (1- \alpha) T_2 ) \textbf{1} + \tilde{\mu}_{12} \frac{\alpha \mathbb{P}_1 + (1- \alpha) \mathbb{P}_2 }{n_1}+ \gamma |u_1 - u_2|^2  \textbf{1}, \label{tau12a}
 \\
 \begin{split}
\mathcal{T}_{21} &= (1- \tilde{\mu}_{21} ) ((1- \varepsilon (1- \alpha)) T_2 + \varepsilon (1- \alpha) T_1)  \textbf{1} \\&+ \tilde{\mu}_{21} \frac{(1- \varepsilon (1- \alpha)) \mathbb{P}_2 + \varepsilon (1- \alpha) \mathbb{P}_1 }{n_2}\\&+ (\frac{1}{3} \varepsilon m_1 (1- \delta)( \frac{m_1}{m_2} \varepsilon ( \delta - 1) + \delta + 1) - \varepsilon \gamma)|u_1 - u_2|^2 \textbf{1}. \label{tau21a}
\end{split}
\end{align}}
If we check the equilibrium distributions (see \cite{Pirner2} for details), we obtain a dependence of  $\tilde{\mu}_{12}$ on $\tilde{\mu}_1$ and of $\tilde{\mu}_{21}$ on $\tilde{\mu}_2$.

An alternative choice to \eqref{tau12a},\eqref{tau21a} is given by
\begin{align}
\mathcal{T}_{12} &= \alpha \frac{\mathbb{P}_1}{n_1} + ( 1 - \alpha) T_2 \textbf{1} + \gamma |u_1 - u_2|^2 \textbf{1}, \label{tau12}
 \\
 \begin{split}
\mathcal{T}_{21} &= (1-\varepsilon (1-\alpha ) )\frac{\mathbb{P}_2}{n_2} + \varepsilon( 1 - \alpha) T_1 \textbf{1} \\&+   ( \frac{1}{3} \varepsilon m_1 (1- \delta)( \frac{m_1}{m_2} \varepsilon ( \delta - 1) + \delta + 1) - \varepsilon \gamma)|u_1 - u_2|^2 \textbf{1}.\end{split} \label{tau21}
\end{align}
This choice of $\mathcal{T}_{12}$ still contains the temperature of gas $1$, since the trace of the pressure tensor $\frac{\mathbb{P}_1}{n_1}$ is the temperature $T_1$.

In \eqref{tau12} compared to \eqref{tau12a} we replace only the temperature $T_1$ of species $1$ by the pressure tensor $\mathbb{P}_1$ while we keep the temperature $T_2$. This asymmetric choice can be motivated by the theory of "persistence of velocity" described  in  chapter \ref{sec7.1.3}. 
We observed in section \ref{sec7.1.3}, that in the one species case the tensor $\mathcal{T}$ can be rewritten as
$$ \mathcal{T} = T \textbf{1} + \tilde{\mu} ~ \text{traceless}[\mathbb{P}], $$
where the memory of the off-equilibrium part of species $1$ is contained in $\tilde{\mu} ~ \text{traceless}[\mathbb{P}].$ 
Doing this analogously for two species we arrive at 
$$\mathcal{T}_{12} = T_{12} \textbf{1} + \frac{\alpha}{n_1} \text{traceless}[\mathbb{P}_1].$$
If we plug in the definition of $T_{12}$ given by \eqref{contemp5}, we end up with \eqref{tau12}.

In both choices of the tensors $\mathcal{T}_{12}$ and $\mathcal{T}_{21}$ we can prove that the models are well-defined and that we have conservation of the number of particles, conservation of total momentum and conservation of total energy. Moreover, we are able to prove an H-Theorem and we can specify the equilibrium distributions. For the proofs see \cite{Pirner2}.

\section*{Conclusion}
The goal in the paper is to derive the four transport coefficients of a two species Navier-Stokes model derived from a two species kinetic model. To this end we take a two species BGK model \cite{Pirner} and are able to perform the Chapman Enskog expansion to derive three of the transport coefficients. A special feature of using model \cite{Pirner} is that this kinetic model has free parameters that can to be chosen a posteriori such that transport coefficients fit with an experiment. After this is done one transport coefficient still needs to be determined. For this we suggest extensions of the BGK model \cite{Pirner} to an ES BGK model \cite{Pirner2}.

\section*{Acknowledgements}
We thank Gabriella Puppo for suggesting this topic to us.

\medskip
Received xxxx 20xx; revised xxxx 20xx.
\medskip


\begin{thebibliography}{99}

\bibitem{AndriesPerthame2001}
Andries P. ,Perthame B., The ES-BGK model equation with correct Prandtl number, AIP conference proceedings 30, 993-1018 (2001)

\bibitem{AndriesAokiPerthame2002}
 Andries P.,  Aoki K. and  Perthame B., A consistent BGK-type model for gas mixtures, Journal of Statistical Physics 106, 993-1018 (2002)
 
\bibitem{Perthame}
Andries P.,  Le Tallec P.,  Perlat J., Perthame B., The Gaussian -BGK model of Boltzmann equation with small Prandtl number, Eur. J. Mech. B - Fluids 19,  813-830 (2000)



\bibitem{Bhatnagar}
 Bhatnagar P. L.,  Gross E. P., and Krook M,
 A model for collision processes in gases. i. small amplitude
  processes in charged and neutral one-component systems,  Physical review 94(3):511 (1954).


\bibitem{Bennoune_2008}
 Bennoune M.,  Lemou M. and  Mieussens L., 
 Uniformly stable numerical schemes for the Boltzmann equation preserving the compressible Navier-Stokes asymptotics, Journal of Computational Physics 227,  3781-3803 (2008)

\bibitem{Bernard_2015}
 Bernard F.,  Iollo A. and  Puppo G., 
 Accurate asymptotic preserving boundary conditions for kinetic equations on Cartesian grids, Journal of Scientific Computing 65, 735-766 (2015)
 
 \bibitem{Bisi_2011}
 Bisi, Marzia, and Giampiero Spiga. "On a kinetic BGK model for slow chemical reactions." Kinet. Relat. Models 4.1 (2011).
 
 \bibitem{Bisi_2017}
 Bisi, Marzia, and Giampiero Spiga. Hydrodynamic limits of kinetic equations for polyatomic and reactive gases. Communications in Applied and Industrial Mathematics 8.1 (2017): 23-42.


 \bibitem{Brull}
  Brull S., An ellipsoidal statistical model for  gas mixtures, Communications in Mathematical Sciences 8, , 1-13 (2015)


%


\bibitem{Crestetto_2012}
 Crestetto A.,  Crouseilles N. and  Lemou M.,
 Kinetic/fluid micro-macro numerical schemes for Vlasov-Poisson-BGK equation using particles,  Kinetic and Related Models 5, 787-816 (2012)

\bibitem{Pirner4}
A. Crestetto, C. Klingenberg, M. Pirner, {\sl Kinetic/fluid micro-macro numerical scheme for a two component plasma,} submitted, 2017.

\bibitem{Chapman}  Chapman S.,  Cowling T. G., The mathematical theory of non-uniform gases, Cambridge University Press (1970)

\bibitem{Dimarco_2014}
 Dimarco G. and  Pareschi L., 
 Numerical methods for kinetic equations, 
Acta Numerica 23, 369-520 (2014)

\bibitem{Jin_2010}
 Filbet F. and  Jin S.,
 A class of asymptotic-preserving schemes for kinetic equations and related problems with stiff sources, Journal of Computational Physics 20, 7625-7648 (2010)




 \bibitem{Groppi}  Groppi M.,  Monica S. and  Spiga G.,  A kinetic ellipsoidal BGK model for a binary gas mixture, epljournal 96, 64002 (2011)
 

 \bibitem{Heer}  Heer C.V., Statistical Mechanics, Kinetic Theory and Stochastic Processes, Elsevier (2012) 
 
 \bibitem{Holway} Holway L., New Statistical Models for Kinetic Theory: Methods of Construction, The Physics of Fluids 9, (1966)
 
 \bibitem{Jeans}  Jeans J.H.,  The persistence of molecular velocities in the kinetic theory of gases, Philosophical Magazine 6, 8:48 (1904), 700-703
 
 \bibitem{Jeans2}  Jeans J.H.,  The Dynamical Theory of Gases, Cambridge University Press (1916)

 
 \bibitem{Pirner}  Klingenberg C., Pirner M., Puppo G., A consistent kinetic model for a two-component mixture with an application to plasma,  Kinetic and related Models 10, 445--465 (2017)
 
   \bibitem{Pirner2}  Klingenberg C.,  Pirner M.,  Puppo G., Kinetic ES-BGK models for a multi-component gas mixture, accepted in Springer Proceedings in Mathematics and Statistics of the International Conference on Hyperbolic Problems: Theory, Numeric and Applications in Aachen 2016 (2017)
   
   \bibitem{Kremer_2006}
   Kremer, Gilberto M., Miriam Pandolfi Bianchi, and Ana Jacinta Soares. "A relaxation kinetic model for transport phenomena in a reactive flow." Physics of Fluids 18.3 (2006): 037104.
 
 \bibitem{Puppo_2007}
 Pieraccini S. and  Puppo G., 
 Implicit-explicit schemes for BGK kinetic equations, Journal of Scientific Computing 32, 1-28 (2007)

 \bibitem{Yun}  Yun S.-B., Classical solutions for the ellipsoidal BGK model with fixed collision frequency, Journal of Differential Equations 259, 6009--6037 (2015)


\end{thebibliography}
\end{document}